\documentclass[dvipsnames,a4paper,draft]{amsart}

\usepackage[dvipsnames]{xcolor} 
\usepackage{amssymb}
\usepackage{wasysym}
\usepackage{marvosym}
\usepackage{enumitem}
\usepackage{tikz}
\usepackage{hyperref}
\usepackage{mathtools}
\usepackage{mathscinet}
\usepackage{graphicx}
\usepackage{psfrag}
\usepackage{datetime}
\usepackage{tikz-cd}
\usepackage{url}
\usepackage{adjustbox}

\usepackage{tabstackengine}
\setstackgap{L}{.75\baselineskip}

\usetikzlibrary{positioning}

\numberwithin{equation}{section}
\newtheorem{theorem}{Theorem}[section]
\newtheorem{lemma}[theorem]{Lemma}
\newtheorem{proposition}[theorem]{Proposition}
\newtheorem{corollary}[theorem]{Corollary}

\theoremstyle{definition}

\newtheorem{example}[theorem]{Example}
\newtheorem{convention}[theorem]{Convention}

\theoremstyle{remark}
\newtheorem{remark}[theorem]{Remark}

\newcommand{\st}{\mid} 

\renewcommand{\emptyset}{\varnothing}

\newcommand{\argpp}[1]{\Pi_{#1}} 
\newcommand{\preproj}{\argpp{n}} 
\newcommand{\argpb}[1]{\overline{\Pi}_{#1}}
\newcommand{\pibar}{\overline{\Pi}_{n}}
\newcommand{\argpris}[1]{\Delta_{#1} \times \Delta_{1}}
\newcommand{\prism}{\argpris{n}} 
\newcommand{\g}[1]{g^{#1}} 
\DeclareMathOperator{\Hom}{Hom}
\newcommand{\kp}{\mathcal{K}_{\preproj}} 
\newcommand{\kpb}{\mathcal{K}_{\pibar}} 
\newcommand{\kpl}{\mathcal{K}_{\Lambda}} 

\newcommand{\tcs}[1]{{\tabbedCenterstack{#1}}}

\newcommand{\sbm}[1]{{\let\amp=&\left(\begin{smallmatrix}#1\end{smallmatrix}\right)}}

\newcommand{\word}[1]{\mathrm{word}(#1)} 

\newcommand{\argsymm}[1]{\mathfrak{S}_{#1}}
\newcommand{\symm}{\argsymm{n + 1}} 

\newcommand{\modules}{\mathsf{mod}\,}
\newcommand{\kbproj}{\mathsf{K}^{\sf b}(\proj \Lambda)}

\newcommand{\twoterm}{\mathsf{K}^{[-1,0]}(\proj \Lambda)}
\newcommand{\proj}{\mathsf{proj}\,}
\DeclareMathOperator{\thick}{thick}

\DeclareMathOperator{\conv}{conv}

\newcommand{\sttilt}{\mathsf{s\tau\mbox{-}tilt}\,}
\newcommand{\isttiltp}{\mathsf{ind\mbox{-}\tau\mbox
{-}rigid\mbox{-}pair}\,}
\newcommand{\twosilt}{\mathsf{2\mbox{-}silt}\,}
\newcommand{\itwosilt}{\mathsf{ind\mbox{-}2\mbox{-}psilt}\,}

\newcommand{\seq}{\mathsf{seq}}
\newcommand{\abseq}{\seq_{n + 1}(a, b)}
\newcommand{\tri}{\mathsf{tri}}
\newcommand{\intsimplex}{\mathsf{int\mbox{-}sim}}

\title[Prisms and preprojective algebras]{Triangulations of prisms and\\ preprojective algebras of type~$A$}
\author{Osamu Iyama}
\email{iyama@ms.u-tokyo.ac.jp}
\urladdr{https://www.ms.u-tokyo.ac.jp/~iyama/index.html}
\address{Graduate School of Mathematical Sciences, The University of Tokyo, 3-8-1 Komaba, Meguro-ku, Tokyo 153-8914, Japan}
\author{Nicholas J. Williams}
\email{nicholas.williams@lancaster.ac.uk}
\urladdr{https://nchlswllms.github.io/}
\address{Department of Mathematics and Statistics, Fylde College, Lancaster University, Lancaster, LA1 4YF, United Kingdom}
\thanks{OI was supported by JSPS Grant-in-Aid for Scientific Research (B) 16H03923, (C) 18K3209 and (S) 15H05738. NJW was supported by a JSPS International Short-Term Postdoctoral Research Fellowship at the University of Tokyo.}

\begin{document}

\begin{abstract}
We show that indecomposable two-term presilting complexes over $\preproj$, the preprojective algebra of $A_{n}$, are in bijection with internal $n$-simplices in the prism $\prism$, the product of an $n$-simplex with a 1-simplex. We show further that this induces a bijection between triangulations of $\prism$ and two-term silting complexes over $\preproj$ such that bistellar flips of triangulations correspond to mutations of two-term silting complexes. These bijections are shown to compatible with the known bijections involving the symmetric group.
\end{abstract}

\maketitle

\section{Introduction}

Cluster algebras are intimately connected with the combinatorics of triangulations. The simplest example of this is the bijection between the clusters in the type~$A$ cluster algebra and triangulations of a convex polygon \cite{fz1}. Further connections were found in the work of Fomin, Shapiro, and Thurston, who defined cluster algebras using tagged triangulations of surfaces \cite{fst}.

Cluster algebras were first connected with the representation theory of finite-dimensional algebras in \cite{mrz}, which led to the definition of cluster categories of hereditary algebras in \cite{bmrrt}. Another approach to categorifying cluster algebras uses the representation theory of preprojective algebras \cite{gls-rigid,gls-aus,gls-mult,gls-flag}. Preprojective algebras have connections with Kleinian singularities \cite{cbh}, Nakajima quiver varieties \cite{nakajima,st}, and crystal bases \cite{ks}.	

Cluster categories were subsequently extended to cluster algebras arising from triangulations of surfaces in \cite{amiot-cluster}. Since then, the relation between triangulated surfaces and representation theory in the form of so-called geometric models has been an active subject of research \cite{bz,ops_geometric,bcs,chang_schroll}.

The triangulations considered thus far are all two-dimensional and it is natural to wonder whether similar phenomena exist in higher dimensions. Indeed, a beautiful connection between representation theory and higher-dimensional triangulations was found in \cite{ot}, where triangulations of even-dimensional cyclic polytopes were shown to be in bijection with cluster-tilting objects in higher cluster categories. In \cite{njw-hst} it was shown how odd-dimensional triangulations enter the picture too, through a bijection with equivalence classes of maximal green sequences.

High-dimensional polytopes are in general very complicated, but, besides cyclic polytopes, another example that is well-studied is the Cartesian product $\Delta_{m} \times \Delta_{n}$ of two simplices. The triangulations of this polytope are interesting for many reasons. The regular triangulations classify the types of tropical polytopes \cite{ds_tc}. Furthermore, the secondary polytope of $\Delta_{m} \times \Delta_{n}$ is the Newton polytope of the product of all minors of an $(m + 1) \times (n + 1)$ matrix \cite{bb_gpm}, as well as the state polytope of the Segre embedding of $\mathbb{P}^{m} \times \mathbb{P}^{n}$ \cite{sturmfels_gbtv}. Products of simplices are also related to transportation polytopes from operations research \cite{lkos}, and arise as strategy spaces for two-player games \cite{lt_simp}.

In this paper, we show that combinatorics of the prism $\prism$ is closely related to the representation theory of $\preproj$, the preprojective algebra of $A_{n}$.

\begin{theorem}[{Proposition~\ref{prop:indec_bij}, Corollary~\ref{cor:triang_bij}, Proposition~\ref{prop:flip}}]
There is a bijection between codimension one internal simplices in $\prism$ and indecomposable two-term presilting complexes over $\preproj$ which induces a bijection between triangulations of $\prism$ and two-term silting complexes over $\preproj$. Mutations of silting complexes correspond to bistellar flips of triangulations.
\end{theorem}

It is already known from \cite[Chapter~7, Section~3C]{gkz-book} that triangulations of $\prism$ are in bijection with permutations in the symmetric group $\symm$, and from \cite{mizuno-preproj} that permutations in $\symm$ were in bijection with two-term silting complexes over $\preproj$. We show that our results are compatible with these bijections. Indeed, to summarise the results of the paper and how they fit in with the literature, there are bijections
\[
\intsimplex_n(\Delta_n\times\Delta_1) \longleftrightarrow \itwosilt\preproj \longleftrightarrow \isttiltp\preproj
\]
which induce the commutative diagram
\[
\begin{tikzpicture}[xscale=4.3,yscale=2]

\node(perm) at (0,0) {$\symm$};
\node(silt) at (-1,-1) {$\twosilt\preproj$};
\node(ttilt) at (-2,-1) {$\sttilt\preproj$};
\node(triangs) at (0,-1) {$\tri(\prism)$};
\node(itr) at (-2,-2) {$P_{n}(\isttiltp\preproj)$};
\node(ipre) at (-1,-2) {$P_{n}(\itwosilt\preproj)$};
\node(simp) at (0,-3) {$P_{n}(\intsimplex_n(\prism))$};

\draw[<->] (perm) -- (ttilt) node[midway,above] {\cite{mizuno-preproj}};
\draw[<->] (ttilt) -- (silt) node[midway,below] {\cite{air}};
\draw[<->] (silt) -- (triangs) node[midway,above] {Corollary~\ref{cor:triang_bij}};
\draw[<->] (triangs) -- (perm);
\draw[->,dashed] (ttilt) -- (itr);
\draw[->,dashed] (silt) -- (ipre);
\draw[->,dashed] (triangs) -- (simp);
\draw[<->] (itr) -- (ipre);
\draw[<->] (ipre) -- (simp);
\draw[<->] (simp) -- (itr);

\node at (-1.45,-1.8) {\cite{air}};
\node at (-0.2,-0.5) {\cite{gkz-book}};
\node at (-0.35,-2.3) {Proposition~\ref{prop:indec_bij}};
\node at (-1.2,-2.7) {Proposition~\ref{prop:indec_bij}};

\end{tikzpicture}
\]
Here
\begin{itemize}
\item $\sttilt\preproj$ is the set of basic support $\tau$-tilting pairs over $\preproj$;
\item $P_{n}(\isttiltp\preproj)$ is the set of size-$n$ sets of indecomposable $\tau$-rigid pairs over $\preproj$;
\item $\twosilt\preproj$ is the set of basic two-term silting complexes over $\preproj$;
\item $P_{n}(\itwosilt\preproj)$ is the set of size-$n$ sets of indecomposable two-term presilting complexes over $\preproj$;
\item $\tri(\prism)$ is the set of triangulations of $\prism$;
\item $P_{n}(\intsimplex_n(\prism))$ is the set of size-$n$ sets of internal $n$-simplices in $\prism$;
\item $\longleftrightarrow$ arrows denote bijections;
\item $\dashrightarrow$ arrows denote direct sum decompositions, or decompositions into internal $n$-simplices.
\end{itemize}

This paper is structured as follows. We give background to the paper in Section~\ref{sect:back}, first covering representation theory and, in particular, $\tau$-tilting theory, preprojective algebras of type~$A$, and their relation with $\symm$. We then cover background on triangulations of $\prism$ and their relation with $\symm$. We prove our results in Section~\ref{sect:res}, first showing the bijection between internal simplices in $\prism$ and indecomposable two-term presilting complexes over $\preproj$. We show how this induces a bijection between triangulations of $\prism$ and two-term silting complexes over~$\preproj$. In order to do this, we make use of a certain factor algebra $\pibar$ of $\preproj$ and prove some intermediate lemmas concerning its representation theory. Finally, we show that our bijections are compatible with the bijections with $\symm$.

\section{Background}\label{sect:back}

\subsection{Representation theory}

In this section, by $\Lambda$ we mean a finite-dimensional algebra over a field $K$, and we write $\modules \Lambda$ for the category of right $\Lambda$-modules. We use $\tau$ to denote the Auslander--Reiten translate in $\modules \Lambda$. We denote by $\kpl := \kbproj$ the homotopy category of bounded complexes in $\proj\Lambda$.

\subsubsection{Support $\tau$-tilting pairs}

$\tau$-tilting theory was introduced in \cite{air} as a generalisation of cluster-tilting theory. A $\Lambda$-module $M$ is called \emph{$\tau$-rigid} if $\Hom_{\Lambda}(M,$ $\tau M) = 0$. A pair $(M, P)$ of $\Lambda$-modules where $P$ is projective is called \emph{$\tau$-rigid} if $M$ is $\tau$-rigid and $\Hom(P, M) = 0$. A $\tau$-rigid pair $(M, P)$ is called \emph{support $\tau$-tilting} if $|M| + |P| = |\Lambda|$, where $|X|$ denotes the number of non-isomorphic indecomposable direct summands of $X$. Here $M$ is called a \emph{support $\tau$-tilting module}. We call a $\tau$-rigid pair \emph{indecomposable} if it is of the form $(M, 0)$ with $M$ indecomposable, or of the form $(0, P)$ with $P$ indecomposable. We refer to $\tau$-rigid pairs of the form $(0, P)$ as \emph{shifted projectives}. We write $\sttilt\Lambda$ for the set of (isomorphism classes of) basic support $\tau$-tilting pairs over $\Lambda$ and $\isttiltp\Lambda$ for the set of (isomorphism classes of) basic indecomposable $\tau$-rigid pairs over $\Lambda$.

\subsubsection{Two-term silting}\label{sect:back:rep:tts}

Two-term silting complexes over $\Lambda$ are in bijection with support $\tau$-tilting pairs over $\Lambda$ and are in many ways nicer to work with \cite{air}. An object $T$ of $\kpl$ is \emph{presilting} if \[\Hom_{\kpl}(T, T[i]) = 0\] for all $i > 0$. A presilting complex $T$ is \emph{silting} if, additionally, $\thick T = \kpl$. Here $\thick T$ denotes the smallest full subcategory of $\kpl$ which contains $P$ and is closed under cones, $[\pm 1]$, direct summands, and isomorphisms. For a two-term complex $T$ to be presilting, it suffices that $\Hom_{\kpl}(T, T[1]) = 0$. Moreover, for a presilting two-term complex $T$ to be silting, it suffices that $|T| = |\Lambda|$ by \cite[Proposition~3.3(b)]{air}. We write $\twosilt\Lambda$ for the set of (isomorphism classes of) two-term silting complexes over $\Lambda$ and $\itwosilt\Lambda$ for the set of (isomorphism classes of) indecomposable two-term presilting complexes over $\Lambda$.

Given two-term silting complexes $T, T' \in \twoterm$, we say that $T'$ is a \emph{mutation} of $T$ if and only if $T = E \oplus X, T' = E \oplus Y$, where $X$ and $Y$ are indecomposable with $X \not\cong Y$.

Adachi, Iyama, and Reiten showed that there was a bijection between $\itwosilt\Lambda$ and $\isttiltp\Lambda$ which induced a bijection between $\twosilt\Lambda$ and $\sttilt\Lambda$ \cite[Theorem~3.2]{air}. Here, if $(M, P)$ is an indecomposable support $\tau$-rigid pair, then $P^{-1} \oplus P \to P^{0}$ is a two-term presilting complex, where $P^{-1} \to P^{0}$ is a minimal projective presentation of $M$.

\subsubsection{Preprojective algebras of type~$A$}

The algebras we are interested in are the preprojective algebras of type~$A$, which are defined as follows using quivers and relations. Preprojective algebras were originally defined by Gel{\cprime}fand and Ponomarev \cite{gp}. They were subsequently generalised to non-simply-laced types in \cite{gls_found}. The preprojective algebra of~$A_{n}$, denoted $\preproj$, has quiver $Q_{n}$ \[
\begin{tikzcd}
1 \ar[r,shift left,"\alpha_{1}"] & 2 \ar[l,shift left,"\beta_{1}"] \ar[r,shift left,"\alpha_{2}"] & 3 \ar[l,shift left,"\beta_{2}"] \ar[r,shift left,"\alpha_{3}"] & \cdots \ar[l,shift left,"\beta_{3}"] \ar[r,shift left,"\alpha_{n - 2}"] & n - 1 \ar[r,shift left,"\alpha_{n - 1}"] \ar[l,shift left,"\beta_{n - 2}"] & n \ar[l,shift left,"\beta_{n - 1}"] 
\end{tikzcd}
\] with relations $\beta_{i}\alpha_{i} = \alpha_{i + 1}\beta_{i + 1}$ for $i \in \{1, 2, \dots, n - 1\}$, and $\alpha_{1}\beta_{1} = \beta_{n - 1}\alpha_{n - 1} = 0$. We compose arrows using the convention $\xrightarrow{\gamma}\xrightarrow{\delta} = \gamma\delta$. We denote the idempotent at vertex $i$ by $e_{i}$ and the indecomposable projective $\preproj$-module at vertex~$i$ by $P_{i}$. The module $P_{i}$ is also the indecomposable injective at vertex $n - i + 1$.

In order to study the $\tau$-tilting theory of $\preproj$, it will be useful to introduce a particular quotient of it. We let \[z = \sum_{i = 1}^{n-1}\alpha_{i}\beta_{i}\] in $\preproj$. This is a central element of $\preproj$. We define
\[\pibar = \preproj/\langle z \rangle.\]
Then $\pibar$ is the quotient of $\preproj$ by the ideal generated by all two-cycles. That is, $\pibar = \preproj/I$, where $I = \langle \alpha_{i}\beta_{i}, \beta_{i}\alpha_{i} \st i \in \{1, 2, \dots, n - 1\} \rangle$. This algebra was considered in \cite{ejr,bcz,dirrt}. Denote by $\overline{P}_{i}$ the projective $\pibar$-module at vertex $i$. By \cite[Theorem~11]{ejr}, the functor $-\otimes_{\preproj}\pibar\colon \kp \to \kpb$ induces a bijection $\twosilt\preproj \to \twosilt\pibar$. Given a two-term silting complex $P^{\bullet}$ of $\preproj$, we write $\overline{P}^{\bullet}$ for the corresponding two-term silting complex over $\pibar$.

\subsubsection{$g$-vectors}

Let $P^{\bullet} = P^{-1} \to P^{0}$ be a two-term complex over projectives over $\preproj$, where $P^{0} = \bigoplus_{i = 1}^{n}P_{i}^{r_{i}}$ and $P^{-1} = \bigoplus_{i = 1}^{n}P_{i}^{s_{i}}$. Then the \emph{$g$-vector} of $P^{\bullet}$ is \[g^{P^{\bullet}} := (r_{1} - s_{1}, r_{2} - s_{2}, \dots, r_{n} - s_{n}).\]
Note that this is the class of $P^{\bullet}$ in the Grothendieck group $K_0(\kp) = K_0(\proj\preproj)$, identified with $\mathbb{Z}^n$.
Given a $\tau$-rigid pair $(M, P)$, we define $g^{(M, P)} := g^{P^{\bullet}}$ where $P^{\bullet}$ is the corresponding two-term presilting complex. It is known that two-term presilting complexes and support $\tau$-rigid pairs are uniquely determined by their $g$-vectors \cite[Theorem~5.5]{air}.

\subsubsection{Relation with permutations}\label{sect:back:rep:perm}

It was shown in \cite{mizuno-preproj} (see also \cite{birs}) that support $\tau$-tilting modules over $\preproj$ were in bijection with permutations in $\symm$. This was generalised to non-simply-laced types in \cite{fg,murakami}. The bijection is constructed as follows. Define $I_{i} := \preproj (1 - e_{i}) \preproj$, the principal two-sided ideal of $\preproj$ generated by $(1 - e_{i})$. Then, for $w \in \symm$ with reduced expression $w = s_{i_{1}}s_{i_{2}} \dots s_{i_{k}}$, define
\[I_{w} = I_{i_{1}}I_{i_{2}} \dots I_{i_{k}}.\]
The map $w \mapsto I_{w}$ then defines a bijection from $\symm$ to the support $\tau$-tilting modules over $\preproj$. In the case $w = e$, the identity permutation, the corresponding support $\tau$-tilting module is the regular module $\preproj$. We write $P_{w}$ for the projective module $P_{w}$ such that $(I_{w}, P_{w})$ is a support $\tau$-tilting pair.

\begin{convention}
We shall use the convention that the base set for the symmetric group $\symm$ is $\{0, 1, \dots, n\}$ and that the simple reflection $s_{i}$ is given by the transposition $(i - 1\ i)$.
\end{convention}

\subsection{Convex polytopes}

We now give background on the side of the paper concerning triangulations of the prism $\prism$.

\subsubsection{Prisms of simplices}

We now explain the relevant background on prisms of simplices, following \cite[Section~6.2.1]{lrs}. The \emph{prism} of the $n$-simplex $\Delta_{n}$ is the polytope $\prism$. The particular geometric realisation of this polytope is not important for the results of this paper, but for the sake of clarity, we take $\prism$ to be the convex hull of the points given by the column vectors of the $(n + 3) \times (2n + 2)$ matrix \[
\begin{pmatrix}
1 & 0 & \dots & 0 & 1 & 0 & \dots & 0 \\
0 & 1 & \dots & 0 & 0 & 1 & \dots & 0 \\
\vdots & \vdots & \ddots & \vdots & \vdots & \vdots & \ddots & \vdots \\
0 & 0 & \dots & 1 & 0 & 0 & \dots & 1 \\
1 & 1 & \dots & 1 & 0 & 0 & \dots & 0 \\
0 & 0 & \dots & 0 & 1 & 1 & \dots & 1
\end{pmatrix}.
\] We label the points given by the first $n + 1$ columns of this matrix by $a_{0}, a_{1}, \dots, a_{n}$ and the points given by the last $n + 1$ columns by $b_{0}, b_{1}, \dots, b_{n}$. We write $V = \{a_{0}, a_{1}, \dots, a_{n}, b_{0}, b_{1}, \dots, b_{n}\}$ for the set of vertices of the prism $\prism$. One can specify $k$-simplices in the prism $\prism$ by giving subsets of $V$ of size $k + 1$, provided the chosen vertices are affinely independent. An \emph{internal} $k$-simplex is one which does not lie in the boundary of $\prism$. We write $\intsimplex_{n}(\prism)$ for the set of internal $n$-simplices of $\prism$.

\subsubsection{Triangulations}

We again follow \cite{lrs}. A \emph{polyhedral subdivision} $\mathcal{S}$ of $\prism$ is a set of $(n + 1)$-dimensional convex polytopes $\{L_{1}, \dots, L_{r}\}$ such that
\begin{itemize}
\item $\bigcup_{i = 1}^{r} L_{i} = \prism$ and
\item $L_{i} \cap L_{j}$ is a face of both $L_{i}$ and $L_{j}$ for all $i, j$.
\end{itemize}
A polyhedral subdivision $\mathcal{T}$ is a \emph{triangulation} if $L_{i}$ is a simplex for every $i$. We write $\tri(\prism)$ for the set of triangulations of $\prism$.

One polyhedral subdivision $\mathcal{S} = \{L_{1}, L_{2}, \dots, L_{r}\}$ \emph{refines} another polyhedral subdivision $\mathcal{S}' = \{L'_{1}, L'_{2}, \dots, L'_{r'}\}$ if for every $L_{i}$ we have $L_{i} \subseteq L'_{j}$ for some $j$, and $\mathcal{S} \neq \mathcal{S}'$. A polyhedral subdivision is an \emph{almost triangulation} if all of its refinements are triangulations. Two triangulations $\mathcal{T}$ and $\mathcal{T}'$ are related by a \emph{bistellar flip} if there is an almost triangulation $\mathcal{S}$ whose only two refinements are $\mathcal{T}$ and $\mathcal{T}'$.

\begin{remark}
Bistellar flips are the generalisation of the operation of flipping a diagonal inside a quadrilateral in a triangulated convex polygon. Here, the almost triangulation is given by a subdivision consisting of triangles and one quadrilateral. The two triangulations of the quadrilateral then give the two refinements of this subdivision which are the triangulations related by a bistellar flip.
\end{remark}

A \emph{circuit} of $\prism$ is a pair $(Z, Z')$ of two disjoint vertex subsets $Z$ and $Z'$ such that $\conv(Z) \cap \conv(Z') \neq \emptyset$ and such that $Z$ and $Z'$ are minimal with respect to this property. Here `$\conv$' denotes the convex hull. The circuits of $\prism$ are given by $(\{a_{i}, b_{j}\}, \{a_{j}, b_{i}\})$ for $i \neq j$.

\subsubsection{Relation with permutations}\label{sect:back:triang:perm}

Triangulations of $\prism$ are in bijection with permutations in the symmetric group $\symm$, as shown in \cite[Chapter~7, Section~3C]{gkz-book}. Indeed, for a permutation $w = i_{0}i_{1} \dots i_{n} \in \symm$, we have that \[\{\{a_{i_{0}}, \dots, a_{i_{j}}, b_{i_{j}}, \dots, b_{i_{n}}\} \st 0 \leqslant j \leqslant n\}\] is the set of $(n + 1)$-simplices of the corresponding triangulation, and all triangulations of $\prism$ arise in this way. For a permutation $w \in \symm$, we write $\mathcal{T}_{w}$ for the corresponding triangulation of $\prism$.

\section{Results}\label{sect:res}

In this section, we show our main result that there is a bijection between the sets $\itwosilt\preproj$ of indecomposable two-term presilting complexes over $\preproj$ and $\intsimplex_{n}(\prism)$ of internal $n$-simplices in $\prism$ which induces a bijection between the sets $\twosilt\preproj$ of two-term silting complexes over $\preproj$ and $\tri(\prism)$ of triangulations of $\prism$.

\subsection{Bijection for simplices}

We begin by showing the first bijection. We write \[\abseq := \{a, b\}^{n + 1} \setminus \{a^{n + 1}, b^{n + 1}\}\] for the set of words of length $n + 1$ in the alphabet $\{a, b\}$ which use both letters. Recall that a \emph{facet} of a polytope is a face of codimension one.

\begin{lemma}\label{prop:simp_desc}
There is a bijection between $\intsimplex_{n}(\prism)$ and $\abseq$.
\end{lemma}
\begin{proof}
Clearly $\{a_{0}, a_{1}, \dots, a_{n}\}$ and $\{b_{0}, b_{1}, \dots, b_{n}\}$ both correspond to facets of $\prism$. Furthermore, $V \setminus \{a_{i}, b_{i}\}$ corresponds to a facet of $\prism$ for all~$i$, since it is the Cartesian product of a facet of $\Delta_{n}$ with $\Delta_{1}$. Thus, an internal $n$-simplex in $\prism$ is given by choosing $a_{i}$ or $b_{i}$ for each $i$, excluding the cases $\{a_{0}, a_{1}, \dots, a_{n}\}$ and $\{b_{0}, b_{1}, \dots, b_{n}\}$. Such a set of vertices is, moreover, affinely independent, and so indeed has an $n$-simplex as its convex hull. Hence, internal $n$-simplices in $\prism$ correspond bijectively to words of length $n + 1$ in the alphabet $\{a, b\}$ which use both letters.
\end{proof}

Given $X \in \abseq$, we write $\Delta_{X}$ for the corresponding internal $n$-simplex in $\intsimplex_{n}(\prism)$.

\begin{proposition}\label{prop:indec_bij}
The sets $\intsimplex_{n}(\prism)$ and $\isttiltp\preproj$ are in bijection with each other.
Moreover, the sets $\intsimplex_{n}(\prism)$ and $\itwosilt\preproj$ are also in bijection with each other.
\end{proposition}
\begin{proof}
We argue in terms of indecomposable $\tau$-rigid pairs. It follows from \cite{mizuno-preproj} that the indecomposable $\tau$-rigid modules over $\preproj$ correspond to submodules of indecomposable injectives---see also \cite{irrt}. Thus, the indecomposable $\tau$-rigid modules with socle $i$ correspond to Young diagrams lying in an $(n + 1 - i) \times i$ grid, rotated so that the north-west corner of the Young diagram is lying at the bottom. Such Young diagrams are determined by the path given by their upper contour. Let us label upwards steps in the path by $a$ and downwards steps in the path by $b$ and orient the paths from left to right. Such Young diagrams correspond to words in the alphabet $\{a, b\}$ with $i$ `$b$'s and $n + 1 - i$ `$a$'s, such that at least one `$a$' precedes a `$b$'. This can be seen in Example~\ref{ex:indec_bij} and Figure~\ref{fig:indec_bij}.

We associate the shifted projective with top $i$ to the word given by $i$ `$b$'s followed by $n + 1 - i$ `$a$'s. Such a word gives an empty Young diagram, and so is not included in the words corresponding to the $\tau$-rigid modules.

We obtain that the indecomposable $\tau$-rigid pairs over $\preproj$ are in bijection with $\abseq$, and so are in bijection with the internal $n$-simplices in $\prism$ by Proposition~\ref{prop:simp_desc}.

By \cite{air}, we obtain the second statement.
\end{proof}

Given a indecomposable $\tau$-rigid pair $(M, 0)$, we write $\word{M, 0}$ for the corresponding word in the alphabet $\{a, b\}$, usually abbreviating this to $\word{M}$. We do likewise for $\word{0, P}$.

\begin{remark}
In \cite{irrt}, it was shown that join-irreducible permutations in $\symm$ were in bijection with $\tau$-rigid modules over $\preproj$. The bijection between $\tau$-rigid modules and $\abseq \setminus \{b^k a^{n + 1 - k} \mid 1 \leqslant k \leqslant n\}$ gives a neat way of seeing this. Indeed, the join-irreducible permutations are the ones with precisely one descent. That is, if the permutation is $i_{0}i_{1}\dots i_{n} \in \symm$, there is some $l \in \{1, 2, \dots, n\}$ such that \[i_{0} < \dots < i_{l - 1} > i_{l} < \dots < i_{n}.\] Such a permutation corresponds to the internal $n$-simplex $\Delta_{X}$ in $\prism$ with word $X = x_{0}x_{1} \dots x_{n}$ where $x_{i_{j}} = a$ for $j \in \{0, 1, \dots, l - 1\}$ and $x_{i_{j}} = b$ otherwise. This then gives a $\tau$-rigid module via Proposition~\ref{prop:indec_bij}. Note that the words $b^k a^{n + 1 - k}$ corresponding to shifted projectives do not arise via this construction, since $i_{l - 1} > i_{l}$ guarantees that at least one `$a$' entry lies to the right of a `$b$' entry.
\end{remark}

\begin{example}\label{ex:indec_bij}
We give the example of the indecomposable $\tau$-rigid pairs of $\argpp{2}$, as shown in Figure~\ref{fig:indec_bij}. We use $[1]$ to denote the indecomposable $\tau$-rigid pairs of the form $(0, P)$. The words $baa$ and $bba$ correspond to the indecomposable $\tau$-rigid pairs \[\left(0, \tcs{2\\1}\right)\text{ and }\left(0, \tcs{1\\2}\right).\] The internal 2-simplices in $\Delta_{2} \times \Delta_{1}$ are shown in Figure~\ref{fig:int_simps}. These internal 2-simplices are in the same position in the figure as their corresponding indecomposable $\tau$-rigid pairs in Figure~\ref{fig:indec_bij}.
\end{example}

\begin{figure}
\caption{Indecomposable $\tau$-rigid pairs for $\argpp{2}$}\label{fig:indec_bij}
 \[
\begin{tikzpicture}[scale=0.66]

\begin{scope}[shift={(-3,0)}]

\node at (0,1) {\bf 1};
\node at (1,2) {\bf 2};

\draw (0,0) -- (-1,1) -- (1,3) -- (2,2) -- (0,0);
\draw (0,2) -- (1,1);
\draw[ultra thick,red] (-1,1) -- (1,3) -- (2,2);

\node at (-0.7,1.7) {\color{blue} $a$};
\node at (0.3,2.7) {\color{blue} $a$};
\node at (1.7,2.7) {\color{blue} $b$};

\end{scope}


\begin{scope}[shift={(-3,-4)}]

\node at (0,1) {\bf 1};
\node at (1,2) {\color{gray} 2};

\draw (0,0) -- (-1,1) -- (1,3) -- (2,2) -- (0,0);
\draw (0,2) -- (1,1);
\draw[ultra thick,red] (-1,1) -- (0,2) -- (1,1) -- (2,2);

\node at (-0.7,1.7) {\color{blue} $a$};
\node at (0.7,1.7) {\color{blue} $b$};
\node at (1.3,1.7) {\color{blue} $a$};

\end{scope}


\begin{scope}[shift={(3,-10)}]

\node at (-1,2) {\bf 1};
\node at (0,1) {\bf 2};

\draw (0,0) -- (-2,2) -- (-1,3) -- (1,1) -- (0,0);
\draw (0,2) -- (-1,1);

\draw[ultra thick,red] (-2,4) -- (-1,3) -- (0,4) -- (1,5);

\node at (-1.3,3.7) {\color{blue} $b$};
\node at (-0.7,3.7) {\color{blue} $a$};
\node at (0.3,4.7) {\color{blue} $a$};

\node at (1.4,1.5) {\LARGE $[1]$};

\end{scope}


\begin{scope}[shift={(3,0)}]

\node at (-1,2) {\bf 1};
\node at (0,1) {\bf 2};

\draw (0,0) -- (-2,2) -- (-1,3) -- (1,1) -- (0,0);
\draw (0,2) -- (-1,1);
\draw[ultra thick,red] (-2,2) -- (-1,3) -- (1,1);

\node at (-1.7,2.7) {\color{blue} $a$};
\node at (-0.3,2.7) {\color{blue} $b$};
\node at (0.7,1.7) {\color{blue} $b$};

\end{scope}


\begin{scope}[shift={(3,-4)}]

\node at (-1,2) {\color{gray} 1};
\node at (0,1) {\bf 2};

\draw (0,0) -- (-2,2) -- (-1,3) -- (1,1) -- (0,0);
\draw (0,2) -- (-1,1);
\draw[ultra thick,red] (-2,2) -- (-1,1) -- (0,2) -- (1,1);

\node at (-1.3,1.7) {\color{blue} $b$};
\node at (-0.7,1.7) {\color{blue} $a$};
\node at (0.7,1.7) {\color{blue} $b$};

\end{scope}


\begin{scope}[shift={(-3,-10)}]

\node at (0,1) {\bf 1};
\node at (1,2) {\bf 2};

\draw (0,0) -- (-1,1) -- (1,3) -- (2,2) -- (0,0);
\draw (0,2) -- (1,1);
\draw[ultra thick,red] (-1,5) -- (1,3) -- (2,4);

\node at (-0.3,4.7) {\color{blue} $b$};
\node at (0.7,3.7) {\color{blue} $b$};
\node at (1.3,3.7) {\color{blue} $a$};

\node at (2.4,1.5) {\LARGE $[1]$};

\end{scope}


\end{tikzpicture}
\]
\end{figure}
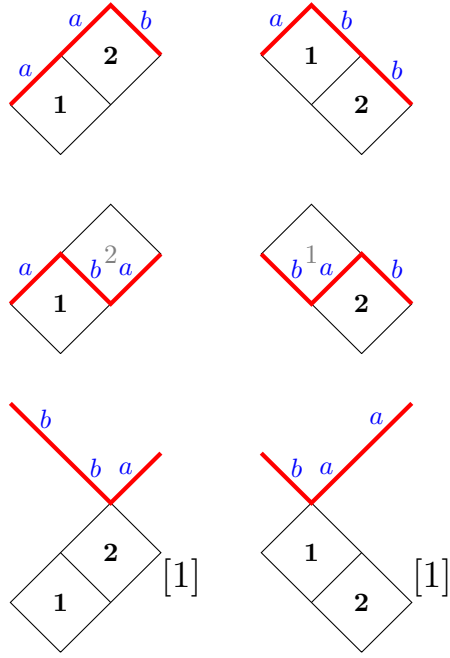

\begin{figure}
\caption{Internal $2$-simplices in $\Delta_{2} \times \Delta_{1}$}\label{fig:int_simps}
\[
\begin{tikzpicture}

\begin{scope}[shift={(-2.5,0)}]

\coordinate (b0) at (0,0);
\coordinate (b1) at (1.5,-1);
\coordinate (b2) at (3,0.5);
\coordinate (a0) at (0,3);
\coordinate (a1) at (1.5,2);
\coordinate (a2) at (3,3.5);

\draw[fill=red,fill opacity=0.5] (a0) -- (a1) -- (b2) -- (a0);

\node [left = 1mm of a0] {$a_{0}$};
\node [below right = 1mm of a1] {$a_{1}$};
\node [below right = 1mm of a2] {$a_{2}$};
\node [left = 1mm of b0] {$b_{0}$};
\node [below right = 1mm of b1] {$b_{1}$};
\node [below right = 1mm of b2] {$b_{2}$};

\draw (a0) -- (a1) -- (a2) -- (a0);
\draw[dashed] (b0) -- (b2);
\draw (b0) -- (b1) -- (b2);
\draw (a0) -- (b0);
\draw (a1) -- (b1);
\draw (a2) -- (b2);

\end{scope}


\begin{scope}[shift={(2.5,0)}]

\coordinate (b0) at (0,0);
\coordinate (b1) at (1.5,-1);
\coordinate (b2) at (3,0.5);
\coordinate (a0) at (0,3);
\coordinate (a1) at (1.5,2);
\coordinate (a2) at (3,3.5);

\draw[fill=red,fill opacity=0.5] (a0) -- (b1) -- (b2) -- (a0);

\node [left = 1mm of a0] {$a_{0}$};
\node [below right = 1mm of a1] {$a_{1}$};
\node [below right = 1mm of a2] {$a_{2}$};
\node [left = 1mm of b0] {$b_{0}$};
\node [below right = 1mm of b1] {$b_{1}$};
\node [below right = 1mm of b2] {$b_{2}$};

\draw (a0) -- (a1) -- (a2) -- (a0);
\draw[dashed] (b0) -- (b2);
\draw (b0) -- (b1) -- (b2);
\draw (a0) -- (b0);
\draw (a1) -- (b1);
\draw (a2) -- (b2);

\end{scope}


\begin{scope}[shift={(-2.5,-6)}]

\coordinate (b0) at (0,0);
\coordinate (b1) at (1.5,-1);
\coordinate (b2) at (3,0.5);
\coordinate (a0) at (0,3);
\coordinate (a1) at (1.5,2);
\coordinate (a2) at (3,3.5);

\draw[fill=red,fill opacity=0.5] (a0) -- (b1) -- (a2) -- (a0);

\node [left = 1mm of a0] {$a_{0}$};
\node [below right = 1mm of a1] {$a_{1}$};
\node [below right = 1mm of a2] {$a_{2}$};
\node [left = 1mm of b0] {$b_{0}$};
\node [below right = 1mm of b1] {$b_{1}$};
\node [below right = 1mm of b2] {$b_{2}$};

\draw (a0) -- (a1) -- (a2) -- (a0);
\draw[dashed] (b0) -- (b2);
\draw (b0) -- (b1) -- (b2);
\draw (a0) -- (b0);
\draw (a1) -- (b1);
\draw (a2) -- (b2);

\end{scope}


\begin{scope}[shift={(2.5,-6)}]

\coordinate (b0) at (0,0);
\coordinate (b1) at (1.5,-1);
\coordinate (b2) at (3,0.5);
\coordinate (a0) at (0,3);
\coordinate (a1) at (1.5,2);
\coordinate (a2) at (3,3.5);

\draw[fill=red,fill opacity=0.5] (b0) -- (a1) -- (b2) -- (b0);

\node [left = 1mm of a0] {$a_{0}$};
\node [below right = 1mm of a1] {$a_{1}$};
\node [below right = 1mm of a2] {$a_{2}$};
\node [left = 1mm of b0] {$b_{0}$};
\node [below right = 1mm of b1] {$b_{1}$};
\node [below right = 1mm of b2] {$b_{2}$};

\draw (a0) -- (a1) -- (a2) -- (a0);
\draw[dashed] (b0) -- (b2);
\draw (b0) -- (b1) -- (b2);
\draw (a0) -- (b0);
\draw (a1) -- (b1);
\draw (a2) -- (b2);

\end{scope}


\begin{scope}[shift={(2.5,-12)}]

\coordinate (b0) at (0,0);
\coordinate (b1) at (1.5,-1);
\coordinate (b2) at (3,0.5);
\coordinate (a0) at (0,3);
\coordinate (a1) at (1.5,2);
\coordinate (a2) at (3,3.5);

\draw[fill=red,fill opacity=0.5] (b0) -- (a1) -- (a2) -- (b0);

\node [left = 1mm of a0] {$a_{0}$};
\node [below right = 1mm of a1] {$a_{1}$};
\node [below right = 1mm of a2] {$a_{2}$};
\node [left = 1mm of b0] {$b_{0}$};
\node [below right = 1mm of b1] {$b_{1}$};
\node [below right = 1mm of b2] {$b_{2}$};

\draw (a0) -- (a1) -- (a2) -- (a0);
\draw[dashed] (b0) -- (b2);
\draw (b0) -- (b1) -- (b2);
\draw (a0) -- (b0);
\draw (a1) -- (b1);
\draw (a2) -- (b2);

\end{scope}


\begin{scope}[shift={(-2.5,-12)}]

\coordinate (b0) at (0,0);
\coordinate (b1) at (1.5,-1);
\coordinate (b2) at (3,0.5);
\coordinate (a0) at (0,3);
\coordinate (a1) at (1.5,2);
\coordinate (a2) at (3,3.5);

\draw[fill=red,fill opacity=0.5] (b0) -- (b1) -- (a2) -- (b0);

\node [left = 1mm of a0] {$a_{0}$};
\node [below right = 1mm of a1] {$a_{1}$};
\node [below right = 1mm of a2] {$a_{2}$};
\node [left = 1mm of b0] {$b_{0}$};
\node [below right = 1mm of b1] {$b_{1}$};
\node [below right = 1mm of b2] {$b_{2}$};

\draw (a0) -- (a1) -- (a2) -- (a0);
\draw[dashed] (b0) -- (b2);
\draw (b0) -- (b1) -- (b2);
\draw (a0) -- (b0);
\draw (a1) -- (b1);
\draw (a2) -- (b2);

\end{scope}


\end{tikzpicture}
\]
\end{figure}
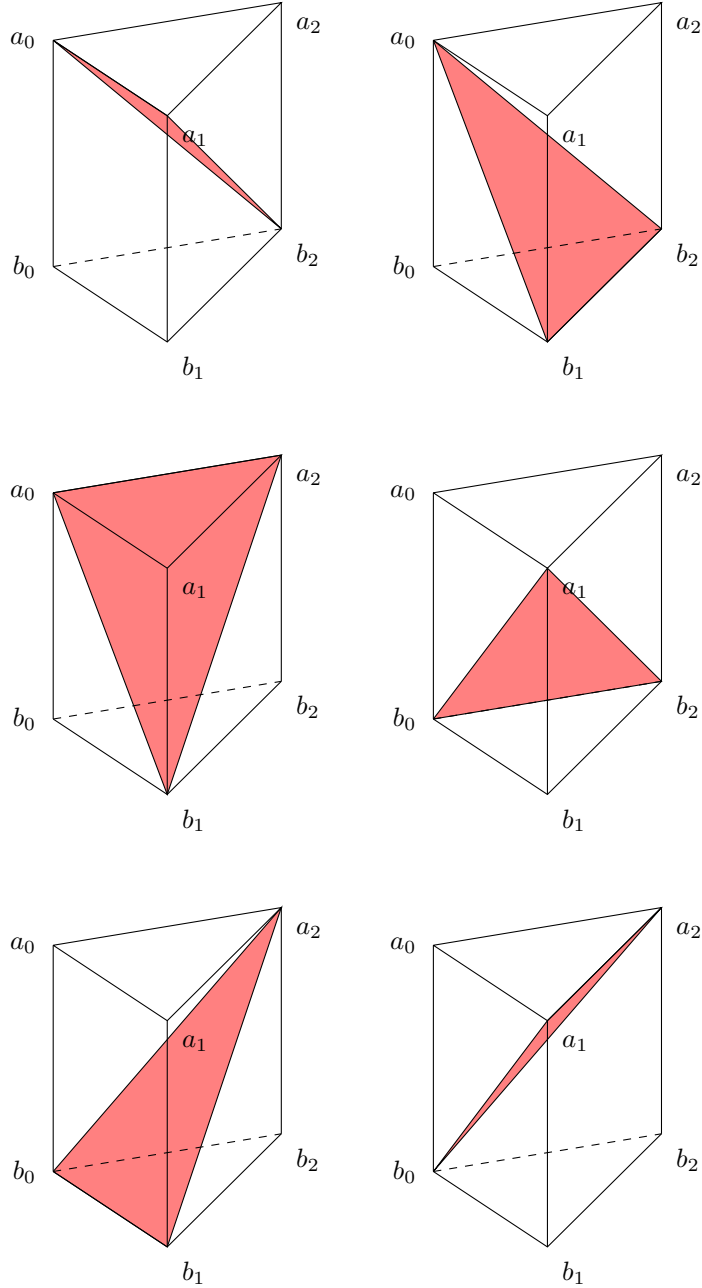

\begin{proposition}\label{prop:g_vec}
Given a $\tau$-rigid pair $(M, P)$ of $\preproj$ corresponding to an internal $n$-simplex $x_{0}x_{1} \dots x_{n}$ in $\prism$, the $g$-vector $\g{(M, P)} = (g_{1}, g_{2}, \dots, g_{n})$ has entries \[
g_{i} = 
	\left\{
		\begin{array}{cl}
		1 & \text{ if } x_{i - 1}x_{i} = ab, \\
		-1 & \text{ if } x_{i - 1}x_{i} = ba, \\
		0 & \text{ otherwise.}
		\end{array}
	\right.
\] 
\end{proposition}
\begin{proof}
Consider an indecomposable $\tau$-rigid module $M$ over $\preproj$ with its dimension vector displayed as in Figure~\ref{fig:indec_bij}. Let $P^{-1} \to P^{0}$ be the projective presentation of $M$. The indecomposable projective with top $i$ is a direct summand of $P^{0}$ if and only if $i$ is a peak in the upper contour of $M$. Similarly, the indecomposable projective with top $i$ is a direct summand of $P^{-1}$ if and only if $i$ is a trough of the upper contour of $M$. Since peaks in the upper contour correspond to $ba$ and troughs correspond to $ab$, we obtain the result.
\end{proof}

\begin{remark}\label{rmk:psilt_map}
Proposition~\ref{prop:g_vec} gives the $g$-vector of the indecomposable two-term silting complex $P^{-1} \to P^{0}$ corresponding to an internal $n$-simplex in $\prism$, and so gives the indecomposable summands of $P^{-1}$ and $P^{0}$. The map $P^{-1} \to P^{0}$ can then be described as follows. If the indecomposable summands of $P^{-1}$ are $P_{i_{1}}, P_{i_{2}}, \dots, P_{i_{l}}$, then $P_{0}$ must either have $l - 1$, $l$, or $l + 1$ summands. The three cases behave quite similarly. If $P^{0}$ has $l + 1$ summands $P_{j_{1}}, \dots, P_{j_{l + 1}}$, then \[j_{1} < i_{1} < j_{2} < i_{2} < \dots < j_{l} < i_{l} < j_{l + 1}.\] Moreover, the component of the map from $P^{-1} \to P^{0}$ from $P_{i_{r}}$ to $P_{j_{s}}$ is zero unless $s \in \{r, r + 1\}$, in which case, the component is the map $P_{i_{r}} \to P_{j_{s}}$ which has as large an image as possible. Such a map is unique up to scalar.
\end{remark}

One can therefore use Proposition~\ref{prop:g_vec} to construct the bijection between the sets $\itwosilt\preproj$ and $\abseq$ directly, without using Proposition~\ref{prop:indec_bij}. We also write $\word{P^{\bullet}}$ for the word corresponding to an indecomposable two-term presilting complex $P^{\bullet}$.

\subsection{Representation theory of $\pibar$}

In this section, we give some results describing the representation theory of $\pibar$ combinatorially, which will then be used to prove the relation between $\preproj$ and triangulations of $\prism$. We first note the following.

\begin{lemma}
Every indecomposable $\pibar$-module is $\tau$-rigid.
\end{lemma}
\begin{proof}
We have from \cite[Proposition~4.1.2]{bcz} that every indecomposable $\pibar$-module is a brick, meaning that every non-zero endomorphism is an isomorphism. It follows from \cite[Theorem~4.1]{dij} that every indecomposable $\pibar$-module must also be $\tau$-rigid, since $\pibar$ is representation-finite by \cite[Proposition~4.1.]{bcz}.
\end{proof}

The indecomposable $\tau$-rigid pairs over $\pibar$ can be described in a way analogous to those of $\preproj$ \cite{dirrt}.

\begin{lemma}\label{lem:pibar_indec}
There is a bijection between $\itwosilt\pibar$, $\isttiltp\pibar$, and $\abseq$, given in the same way as in Proposition~\ref{prop:indec_bij} and Proposition~\ref{prop:g_vec}.
\end{lemma}
\begin{proof}
This bijection follows from Proposition~\ref{prop:g_vec} and \cite[Theorem~11]{ejr}, which says that the $g$-vectors of indecomposable presilting complexes for $\pibar$ must be the same as those for $\preproj$. The bijection between $\abseq$ and $\isttiltp\pibar$ then follows from \cite{air}, and it is clear that it can be constructed in a similar way to Proposition~\ref{prop:indec_bij}.
\end{proof}

For indecomposable $\tau$-rigid $\pibar$-modules $M$ and indecomposable presilting complexes $\overline{P}^{\bullet}$ over $\pibar$, we write $\word{M}$ and $\word{\overline{P}^{\bullet}}$ for the corresponding words in $\abseq$, as we do with $\preproj$.

\begin{example}\label{ex:pibar_mod}
Consider the $\argpb{6}$-module $M$ given by \[\tcs{\phantom{1}2\phantom{3}\phantom{4}\phantom{5}\phantom{6}\\1\phantom{2}3\phantom{4}\phantom{5}\phantom{6}\\\phantom{1}\phantom{2}\phantom{3}4\phantom{5}6\\\phantom{1}\phantom{2}\phantom{3}\phantom{4}5\phantom{6}}\, .\] Then $\word{M} = aabbbab$, using Lemma~\ref{lem:pibar_indec} and Proposition~\ref{prop:indec_bij}. Furthermore, using Lemma~\ref{lem:pibar_indec} and Proposition~\ref{prop:g_vec}, the projective presentation of $M$ is given by $\overline{P}_{5} \to \overline{P}_{2} \oplus \overline{P}_{6}$.

Similarly, the $\argpb{6}$-module $N$ given by \[\tcs{2\phantom{3}4\\\phantom{2}3\phantom{4}}\] has word $\word{N} = bababaa$ and projective presentation $\overline{P}_{1} \oplus \overline{P}_{3} \oplus \overline{P}_{5} \to \overline{P}_{2} \oplus \overline{P}_{4}$.
\end{example}

The following criterion for identifying the support of an indecomposable $\pibar$-module from its word is very useful.

\begin{lemma}\label{lem:pibar_comp_factors}
An indecomposable $\pibar$-module $M$ with $\word{M} = x_{0}x_{1} \dots x_{n}$ is supported at vertex $j$ if and only if there exist $i$ and $k$ such that $i \leqslant j < k$ such that $x_{i} = a$ and $x_{k} = b$.
\end{lemma}
\begin{proof}
It suffices to note that the composition factors of $M$ always form an interval $[i + 1, k]$ in $[n]$ and that the first of these composition factors $i + 1$ will occur when $x_{i}$ is the first occurrence of `$a$' and the last composition factor $k$ will occur when $x_{k}$ is the last occurrence of `$b$'. Before the first occurrence of `$a$' and after the first occurrence of `$b$', the contour of the module $M$ will go along the boundary of the Young diagram, and so $M$ will not be supported at the corresponding simple module. Between these occurrence, the contour of $M$ never reaches the boundary of the Young diagram again, so $M$ is supported at all of the simples.
\end{proof}

We can describe the Auslander--Reiten translate over $\pibar$. This can also be seen from the interpretation of the Auslander--Reiten translate for string algebras in terms of adding and removing hooks and cohooks from \cite{br_ar}.

\begin{lemma}\label{lem:pibar_tau}
Let $M$ be an indecomposable non-projective $\pibar$-module with \[\word{M} = a^{r} \, b \, w \, a \, b^{s},\] where $w$ is some arbitrary, possibly empty, subword, and possibly $r = 0$ and possibly $s = 0$. Then \[\word{\tau M} = b^{r} \, a \, w \, b \, a^{s}.\]
\end{lemma}
\begin{proof}
Let $P^{-1} \to P^{0} \to M \to 0$ be the projective presentation of $M$. Then there is an exact sequence $0 \to \tau M \to \nu P^{-1} \to \nu P^{0}$. Note that $\nu$ sends a projective at a given vertex to the injective at the same vertex. Hence, $\tau M$ has an injective presentation by the injectives at the same vertices as the projectives in the projective presentation of $M$. Using the dual of Lemma~\ref{lem:pibar_indec}, we conclude that the lower contour of $\tau M$ is given by $\word{M}$. Hence, $\word{\tau M}$ is obtained from $\word{\tau M}$ as shown in Figure~\ref{fig:pibar_tau}, which establishes the result. (The cases where only one of $r$ and $s$ is non-zero are easily extrapolated from those shown in the figure.)
\end{proof}

\begin{example}\label{ex:ex_pibar_tau}
We give an example where we compute the Auslander--Reiten translate of an indecomposable $\pibar$-module. We consider the $\argpb{5}$-module $M$ given by \[\tcs{\phantom{1}2\phantom{3}4\\1\phantom{2}3\phantom{4}}.\] Then $\word{M} = aababa$, so $\word{\tau M} = bbaabb$. Hence $\tau M$ is the module \[\tcs{\phantom{3}4\phantom{5}\\3\phantom{4}5}.\] This can be visualised in Figure~\ref{fig:ex_pibar_tau}.
\end{example}

\begin{figure}
\caption{Illustration of Lemma~\ref{lem:pibar_tau}}\label{fig:pibar_tau}
\[
\begin{tikzpicture}[scale=0.4]

\begin{scope}[shift={(-8,0)}]


\draw (-4,0) -- (-2,-2) -- (0,0) -- (2,-2) -- (4,0);


\draw (-4,2) -- (-2,0) -- (0,2) -- (2,0) -- (4,2);


\draw[red] (-4,0) -- (-5,1);
\draw[red] (4,0) -- (5,1);


\draw[red] (-4,2) -- (-5,1);
\draw[red] (4,2) -- (5,1);


\draw[blue] (-5,1) -- (-7,-1);
\draw[blue] (-5,1) -- (-7,3);


\draw[blue] (5,1) -- (7,-1);
\draw[blue] (5,1) -- (7,3);


\node at (-4.7,1.8) {\color{red} $a$};
\node at (4.7,1.8) {\color{red} $a$};

\node at (-4.7,0.25) {\color{red} $b$};
\node at (4.7,0.25) {\color{red} $b$};

\node at (0,2.5) {$w$};
\node at (0,-1) {$w$};

\node at (-5.8,-0.5) {\color{blue} $a^{r}$};
\node at (-5.8,2.5) {\color{blue} $b^{r}$};

\node at (5.8,-0.5) {\color{blue} $b^{s}$};
\node at (5.8,2.5) {\color{blue} $a^{s}$};


\node at (0,-4) {$r \neq 0$ and $s \neq 0$ case};

\end{scope}

\begin{scope}[shift={(8,0)}]


\draw (-6,0) -- (-4,-2) -- (-2,0) -- (0,-2) -- (2,0) -- (4,-2) -- (6,0);


\draw (-6,2) -- (-4,0) -- (-2,2) -- (0,0) -- (2,2) -- (4,0) -- (6,2);


\draw[red] (-6,0) -- (-7,1);
\draw[red] (6,0) -- (7,1);


\draw[red] (-6,2) -- (-7,1);
\draw[red] (6,2) -- (7,1);


\node at (-6.7,1.8) {\color{red} $a$};
\node at (6.7,1.8) {\color{red} $a$};

\node at (-6.7,0.25) {\color{red} $b$};
\node at (6.7,0.25) {\color{red} $b$};

\node at (0,0.7) {$w$};
\node at (0,-2.5) {$w$};


\node at (0,-4) {$r = 0$ and $s = 0$ case};

\end{scope}

\end{tikzpicture}
\]
\end{figure}
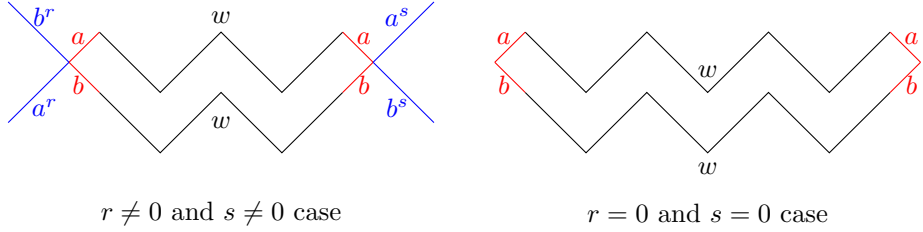

\begin{figure}
\caption{Illustration of Example~\ref{ex:ex_pibar_tau}}\label{fig:ex_pibar_tau}
\[
\begin{tikzpicture}[scale=0.75]

\draw[red] (0,0) -- (2,2) -- (3,1) -- (4,2) -- (5,1) -- (6,2);

\node at (1,0) {\color{red} 1};
\node at (2,1) {\color{red} 2};
\node at (3,0) {\color{red} 3};
\node at (4,1) {\color{red} 4};

\draw[blue] (0,4) -- (2,2) -- (4,4) -- (6,2);

\node at (3,2) {\color{blue} 3};
\node at (4,3) {\color{blue} 4};
\node at (5,2) {\color{blue} 5};

\end{tikzpicture}
\]
\end{figure}
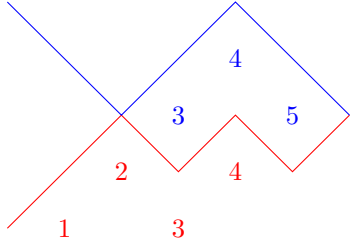

The following description of submodules and factor modules of indecomposable $\pibar$-modules is key. Here, \[X = u\overset{i}{b}v\] means that $b = x_{i}$, where $X = x_{0} \dots x_{n}$. We use $\argpb{[i + 1, j]}$ to denote the algebra $\argpb{j - i}$ with the vertex set given by $[i + 1, j] := \{i + 1, i + 2, \dots, j\}$, in the natural way.

\begin{lemma}\label{lem:pibar_sub_fac}
Let $M$ be an indecomposable $\pibar$-module. Then $M$ has a submodule supported on the vertices $\{i + 1, \dots, j\}$ if and only if one of the following four cases hold.
\begin{enumerate}[label={\upshape(\arabic*)}]
\item $\word{M} = b^{s}\,a\,u\,\overset{i}{b}\,w\,\overset{j}{a}\,v\,b\,a^{r}$, where possibly $r = 0$ or $s = 0$.
\item $\word{M} = b^{i}\,\overset{i}{a}\,w\,\overset{j}{a}\,v\,b\,a^{r}$, where possibly $r = 0$ or $i = 0$.
\item $\word{M} = b^{s}\,a\,u\,\overset{i}{b}\,w\,\overset{j}{b}\,a^{n - j}$, where possibly $s = 0$ or $j = n$.
\item $\word{M} = b^{i}\,\overset{i}{a}\,w\,\overset{j}{b}\,a^{n - j}$, where possibly $i = 0$ or $j = n$.
\end{enumerate}
In each case, if $L$ is the $\argpb{[i + 1, j]}$-module corresponding to the $\pibar$-submodule of $M$, then $\word{L} = awb$. Note that $w$ is possibly empty, as are $u$ and $v$.

Dually, $M$ has a factor module supported on the vertices $\{i + 1, \dots, j\}$ if and only if one of the following four cases hold.
\begin{enumerate}[label={\upshape(\arabic*)}]
\item $\word{M} = b^{s}\,a\,u\,\overset{i}{a}\,w\,\overset{j}{b}\,v\,b\,a^{r}$, where possibly $r = 0$ or $s = 0$.\label{op:1}
\item $\word{M} = b^{i}\,\overset{i}{a}\,w\,\overset{j}{b}\,v\,b\,a^{r}$, where possibly $r = 0$ or $i = 0$.\label{op:2}
\item $\word{M} = b^{s}\,a\,u\,\overset{i}{a}\,w\,\overset{j}{b}\,a^{n - j}$, where possibly $s = 0$ or $j = n$.\label{op:3}
\item $\word{M} = b^{i}\,\overset{i}{a}\,w\,\overset{j}{b}\,a^{n - j}$, where possibly $i = 0$ or $j = n$.\label{op:4}
\end{enumerate}
In each case, if $N$ is the $\argpb{[i + 1, j]}$-module corresponding to the $\pibar$-submodule of $M$, then $\word{N} = awb$. Again, $w$ is possibly empty, as are $u$ and $v$; and possibly $r = 0$ or $s = 0$.
\end{lemma}
\begin{proof}
We focus on the submodule case. The four cases correspond as follows.
\begin{enumerate}
\item $i \neq 0$ and $j \neq n$ and $M$ is supported on $i$ and $j + 1$.
\item $i = 0$ or $M$ is not supported on $i$, but $j \neq n$ and $M$ is supported on $j + 1$.
\item $i \neq 0$ and $M$ is supported at $i$, but $j = n$ or $M$ is not supported on $j + 1$.
\item $i = 0$ or $M$ is not supported at $i$, and $j = n$ or $M$ is not supported at $j + 1$.
\end{enumerate}
These cases are clearly exhaustive. We deal with them one by one. We let $\word{M} = x_{0}x_{1} \dots x_{n}$.

Case~\ref{op:1}: It follows from Lemma~\ref{lem:pibar_comp_factors} that if $M$ is supported on $i$ and $j + 1$, then `$a$' must occur before $i$ and `$b$' must occur after $j$. Since the vector subspace $L$ of $M$ supported on $\{i + 1, \dots, j\}$ is a submodule, then we must have $x_{i} = b$ and $x_{j} = a$ if $L$ is to be closed under the action of $\pibar$.

Case~\ref{op:2}: As in case~\ref{op:1}, we conclude that `$b$' occurs after $j$, and that $x_{j} = a$. If $i = 0$, then $s = 0$; moreover, $x_{0} = x_{i} = a$ by Lemma~\ref{lem:pibar_comp_factors}, since $M$ is supported on $i + 1$ by assumption. If $i \neq 0$, then, since $M$ is not supported on $i$, by Lemma~\ref{lem:pibar_comp_factors}, we must have that $x_{0} \dots x_{i - 1} = b^{i}$.

Cases~\ref{op:3} and~\ref{op:4} can be shown in a similar way to case~\ref{op:2}. In all cases, we have that $\word{L} = awb$ by Lemma~\ref{lem:pibar_comp_factors}, since $L$ is sincere by assumption, and must have the shape prescribed by the word $w$.

The dual cases concerning a factor module of $M$ follow similarly. Case~\ref{op:1} is straightforwardly dual. In case~\ref{op:2}, one can reason that $x_{0} \dots x_{i} = b^{i}a$ in the same way to in the submodule case. Showing that $x_{j} \dots x_{n} = vba^{s}$ is then done in the same way as in the factor module case~\ref{op:1}. The remaining cases can be done in an analogous way.
\end{proof}

We can now prove a criterion for non-vanishing of $\Hom_{\pibar}(N, \tau M) \neq 0$ for indecomposable $\pibar$-modules $M$ and $N$.

\begin{lemma}\label{lem:pibar_ext}
Let $M$ and $N$ be indecomposable $\pibar$-modules with $\word{M} = x_{0}x_{1} \dots x_{n}$ and $\word{N} = y_{0}y_{1} \dots y_{n}$. Then $\Hom_{\pibar}(N, \tau M) \neq 0$ if and only if there exist $i$ and $j$ such that $x_{i} = b$, $x_{j} = a$, $y_{i} = a$, and $y_{j} = b$.
\end{lemma}
\begin{proof}
First note that if $M$ is projective, then $\tau M = 0$, so $\Hom_{\pibar}(N, \tau M) = 0$. In this case, $\word{M} = a^{l}b^{n - l + 1}$ for some $l \in \{1, 2, \dots, n\}$, so it is clear that there can be no $i < j$ with $x_{i} = b$ and $x_{j} = a$.

Excluding the case where $M$ is projective, we have that $\Hom_{\pibar}(N, \tau M) \neq 0$ if and only if there is a non-zero factor module of $N$ which is a submodule of $\tau M$, namely the image of the non-zero map. Hence, this is the case if and only if $\word{\tau M}$ falls into one of the submodule cases of Lemma~\ref{lem:pibar_sub_fac} and $\word{N}$ falls into one of the factor module cases. Hence, there exist $i$ and $j$ such that $y_{i} = a$ and $y_{j} = b$, whilst $\word{\tau M}$ falls into one of the following cases.
\begin{enumerate}
\item $\word{\tau M} = b^{s}\,a\,u\,\overset{i}{b}\,w\,\overset{j}{a}\,v\,b\,a^{r}$, where possibly $s = 0$ or $r = 0$.
\item $\word{\tau M} = b^{i}\,\overset{i}{a}\,w\,\overset{j}{a}\,v\,b\,a^{r}$, where possibly $i = 0$ or $r = 0$.
\item $\word{\tau M} = b^{s}\,a\,u\,\overset{i}{b}\,w\,\overset{j}{b}\,a^{n - j}$, where possibly $s = 0$ or $j = n$.
\item $\word{\tau M} = b^{i}\,\overset{i}{a}\,w\,\overset{j}{b}\,a^{n - j}$, where possibly $i = 0$ or $j = n$.
\end{enumerate}
Applying Lemma~\ref{lem:pibar_tau}, we have these four cases holds if and only if the corresponding one of the following four options hold for $\word{M}$.
\begin{enumerate}
\item $\word{M} = a^{s}\,b\,u\,\overset{i}{b}\,w\,\overset{j}{a}\,v\,a\,b^{r}$, where possibly $s = 0$ or $r = 0$.
\item $\word{M} = a^{i}\,\overset{i}{b}\,w\,\overset{j}{a}\,v\,a\,b^{r}$, where possibly $i = 0$ or $r = 0$.
\item $\word{M} = a^{s}\,b\,u\,\overset{i}{b}\,w\,\overset{j}{a}\,b^{n - j}$, where possibly $s = 0$ or $j = n$.
\item $\word{M} = a^{i}\,\overset{i}{b}\,w\,\overset{j}{a}\,b^{n - j}$, where possibly $i = 0$ or $j = n$.
\end{enumerate}
Thus, if any of these four cases hold, then we have $x_{i} = b$ and $x_{j} = a$.

Conversely, if we have $x_{i} = b$, $x_{j} = a$, $y_{i} = a$, and $y_{j} = b$, then we can choose $i$ and $j$ as close together as possible. It follows that $x_{i} \dots x_{j} = bwa$ and $y_{i} \dots y_{j} = awb$ for some common, possibly empty, subword $w$, otherwise we can find some $i$ and $j$ closer together. Then $\word{M}$ must lie in one of the cases above and $\word{N}$ must lie in one of the factor module cases in Lemma~\ref{lem:pibar_sub_fac}. One can then apply Lemma~\ref{lem:pibar_tau} and the logic runs backwards.
\end{proof}

The proof is related to the description of maps between string modules in \cite{cb_map}.

\subsection{Bijection for triangulations}

The next result shows how extensions between indecomposable presilting complexes are encoded in the words of the corresponding internal $n$-simplices.

\begin{proposition}\label{prop:ext}
Suppose that $P^{\bullet}$ and $Q^{\bullet}$ are indecomposable two-term presilting complexes over $\preproj$, where $\word{P^{\bullet}} = X = x_{0}x_{1} \dots x_{n}$ and $\word{Q^{\bullet}} = Y = y_{0}y_{1} \dots y_{n}$. Then $\Hom_{\kp}(P^{\bullet}, Q^{\bullet}[1]) \neq 0$ if and only if there exist $i$ and $j$ with $i < j$ such that $x_{i} = b$, $x_{j} = a$, $y_{i} = a$, and $y_{j} = b$.
\end{proposition}
\begin{proof}
By \cite[Theorem~11]{ejr}, we have that $\Hom_{\kp}(P^{\bullet}, Q^{\bullet}) \neq 0$ if and only if $\Hom_{\kpb}(\overline{P}^{\bullet}, \overline{Q}^{\bullet}) \neq 0$, so we reason in terms of $\pibar$ instead. Let $(M, M')$ and $(N, N')$ be the indecomposable $\tau$-rigid pairs over $\pibar$ corresponding to $\overline{P}^{\bullet}$ and $\overline{Q}^{\bullet}$.

We first deal with the cases where either $M = 0$ or $N = 0$. If $N = 0$, then $\overline{Q}^{\bullet} = (\overline{Q}^{-1} \to 0)$ and $Y = b^{k}a^{n + 1 - k}$. Then $\Hom_{\kpb}(\overline{P}^{\bullet}, \overline{Q}^{\bullet}[1]) = 0$ and there can be no $i < j$ with $y_{i} = a$ and $y_{j} = b$, which solves this case.

If $M = 0$ and $N \neq 0$, then, by \cite{air}, we have that $\Hom_{\kpb}(\overline{P}^{\bullet}, \overline{Q}^{\bullet}) \neq 0$ if and only if $\Hom_{\pibar}(M', N) \neq 0$. We have that $M' = \overline{P}_{k}$ for some $k$. Hence $\Hom_{\pibar}(M', N) \neq 0$ if and only if $N$ is supported at the vertex $k$, which is the case if and only if there exist $i \leqslant k < j$ such that $y_{i} = a$ and $y_{j} = b$. Since $X = \word{0, \overline{P}_{k}} = b^{k}a^{n + 1 - k}$, we have that $x_{i} = b$ and $x_{j} = a$, as desired.

We can now assume that $M' = 0$ and $N' = 0$, so that $\Hom_{\kpb}(\overline{P}^{\bullet}, \overline{Q}^{\bullet}) \neq 0$ if and only if $\Hom_{\kpb}(N, \tau M) \neq 0$. By Lemma~\ref{lem:pibar_ext}, this is the case if and only if there exist $i$ and $j$ with $i < j$ such that $x_{i} = b$, $x_{j} = a$, $y_{i} = a$, and $y_{j} = b$.
\end{proof}

This proposition therefore gives a simple criterion for the existence of an extension between indecomposable presilting complexes over $\preproj$ and $\pibar$.

\begin{corollary}
Suppose that $P^{\bullet}$ and $Q^{\bullet}$ are indecomposable two-term presilting complexes over $\preproj$ corresponding to internal $n$-simplices $\Delta_{X}$ and $\Delta_{Y}$ in $\prism$. Then $P^{\bullet} \oplus Q^{\bullet}$ is presilting if and only if $\Delta_{X}$ and $\Delta_{Y}$ do not intersect in their interiors.
\end{corollary}
\begin{proof}
The circuits of $\prism$ correspond to $(\{a_{i}, b_{j}\}, \{a_{j}, b_{i}\})$ where $i \neq j$. These are crossing diagonals in the face $\{a_{i}, a_{j}, b_{i}, b_{j}\}$ of $\prism$. Since internal $n$-simplices $\Delta_{x_{0}x_{1} \dots x_{n}}$ and $\Delta_{y_{0}y_{1} \dots y_{n}}$ intersect in their interiors precisely if each contains one half of a circuit, we obtain that these simplices intersect if and only if there exist $i$ and $j$ such that $x_{i} = y_{j} = a$ and $x_{j} = y_{i} = b$. The result then follows by Proposition~\ref{prop:ext}.
\end{proof}

\begin{corollary}\label{cor:triang_bij}
There are bijections between $\tri(\prism)$, $\twosilt\preproj$, and $\sttilt\preproj$.
\end{corollary}
\begin{proof}
It follows from \cite[Proposition~3.3]{air} and \cite{aihara} that two-term silting complexes over $\preproj$ are precisely two-term presilting complexes over $\preproj$ with $n$ non-isomorphic indecomposable summands. Hence, by Proposition~\ref{prop:indec_bij} and Proposition~\ref{prop:ext}, we have that two-term silting complexes over $\preproj$ correspond to collections of $n$ internal $n$-simplices in $\prism$ which do not intersect each other's interiors. It follows from \cite[Chapter~7, Proposition~3.10(a)]{gkz-book} that triangulations of $\prism$ correspond to sets of $n$ internal $n$-simplices which do not intersect each other's interiors.
\end{proof}

Given a support $\tau$-tilting pair $(M, P)$ over $\preproj$, we write $\mathcal{T}(M, P)$ for the corresponding triangulation of $\prism$. Likewise, we write $\mathcal{T}(P^{\bullet})$ for the triangulation of $\prism$ corresponding to a two-term silting complex $P^{\bullet}$ over $\preproj$.

\begin{example}
In each row of Figure~\ref{fig:big}, we show a support $\tau$-tilting pair over $\argpp{2}$, a two-term silting complex over $\argpp{2}$, a permutation in $\argsymm{3}$, a pair of three-letter words in the alphabet $\{a, b\}$, and a triangulation of $\argpris{2}$, all of which correspond to each other under the bijections.
\end{example}

\begin{figure}
\caption{Support $\tau$-tilting pairs, two-term silting complexes, permutations, words, and triangulations of $\Delta_{2} \times \Delta_{1}$}\label{fig:big}
\[
\begin{tikzpicture}[scale=0.95,xscale=0.95,yscale=0.9]

\begin{scope}[shift={(0,0)},scale=0.6]

\coordinate (b0) at (0,0);
\coordinate (b1) at (1.5,-1);
\coordinate (b2) at (3,0.5);
\coordinate (a0) at (0,3);
\coordinate (a1) at (1.5,2);
\coordinate (a2) at (3,3.5);

\draw[fill=red,fill opacity=0.5] (a0) -- (a1) -- (b2) -- (a0);
\draw[fill=blue,fill opacity=0.5] (a0) -- (b1) -- (b2) -- (a0);

\node [left = 1mm of a0] {$a_{0}$};
\node [below right = 1mm of a1] {$a_{1}$};
\node [below right = 1mm of a2] {$a_{2}$};
\node [left = 1mm of b0] {$b_{0}$};
\node [below right = 1mm of b1] {$b_{1}$};
\node [below right = 1mm of b2] {$b_{2}$};

\draw (a0) -- (a1) -- (a2) -- (a0);
\draw[dashed] (b0) -- (b2);
\draw (b0) -- (b1) -- (b2);
\draw (a0) -- (b0);
\draw (a1) -- (b1);
\draw (a2) -- (b2);


\node at (-3,1.5) {$aab, abb$};


\node at (-6,1.5) {$123$};


\node at (-11,1.5) {$\left(0 \to \tcs{\phantom{1}2\\1\phantom{2}}\right) \oplus \left(0 \to \tcs{1\phantom{2}\\\phantom{1}2}\right)$};


\node at (-17,1.5) {$\left(\tcs{\phantom{1}2\\1\phantom{2}} \oplus \tcs{1\phantom{2}\\\phantom{1}2}, \, 0 \right)$};

\end{scope}


\begin{scope}[shift={(0,-4)},scale=0.6]

\coordinate (b0) at (0,0);
\coordinate (b1) at (1.5,-1);
\coordinate (b2) at (3,0.5);
\coordinate (a0) at (0,3);
\coordinate (a1) at (1.5,2);
\coordinate (a2) at (3,3.5);

\draw[fill=red,fill opacity=0.5] (a0) -- (a1) -- (b2) -- (a0);
\draw[fill=blue,fill opacity=0.5] (b0) -- (a1) -- (b2) -- (b0);

\node [left = 1mm of a0] {$a_{0}$};
\node [below right = 1mm of a1] {$a_{1}$};
\node [below right = 1mm of a2] {$a_{2}$};
\node [left = 1mm of b0] {$b_{0}$};
\node [below right = 1mm of b1] {$b_{1}$};
\node [below right = 1mm of b2] {$b_{2}$};

\draw (a0) -- (a1) -- (a2) -- (a0);
\draw[dashed] (b0) -- (b2);
\draw (b0) -- (b1) -- (b2);
\draw (a0) -- (b0);
\draw (a1) -- (b1);
\draw (a2) -- (b2);


\node at (-3,1.5) {$aab, bab$};


\node at (-6,1.5) {$213$};


\node at (-11,1.5) {$\left(0 \to \tcs{\phantom{1}2\\1\phantom{2}}\right) \oplus \left(\tcs{1\phantom{2}\\\phantom{1}2} \to \tcs{\phantom{1}2\\1\phantom{2}}\right)$};


\node at (-17,1.5) {$\left(\tcs{\phantom{1}2\\1\phantom{2}} \oplus 2, \, 0 \right)$};

\end{scope}


\begin{scope}[shift={(0,-8)},scale=0.6]

\coordinate (b0) at (0,0);
\coordinate (b1) at (1.5,-1);
\coordinate (b2) at (3,0.5);
\coordinate (a0) at (0,3);
\coordinate (a1) at (1.5,2);
\coordinate (a2) at (3,3.5);

\draw[fill=red,fill opacity=0.5] (a0) -- (b1) -- (b2) -- (a0);
\draw[fill=blue,fill opacity=0.5] (a0) -- (b1) -- (a2) -- (a0);

\node [left = 1mm of a0] {$a_{0}$};
\node [below right = 1mm of a1] {$a_{1}$};
\node [below right = 1mm of a2] {$a_{2}$};
\node [left = 1mm of b0] {$b_{0}$};
\node [below right = 1mm of b1] {$b_{1}$};
\node [below right = 1mm of b2] {$b_{2}$};

\draw (a0) -- (a1) -- (a2) -- (a0);
\draw[dashed] (b0) -- (b2);
\draw (b0) -- (b1) -- (b2);
\draw (a0) -- (b0);
\draw (a1) -- (b1);
\draw (a2) -- (b2);


\node at (-3,1.5) {$aba, abb$};


\node at (-6,1.5) {$132$};


\node at (-11,1.5) {$\left(\tcs{\phantom{1}2\\1\phantom{2}} \to \tcs{1\phantom{2}\\\phantom{1}2}\right) \oplus \left(0 \to \tcs{1\phantom{2}\\\phantom{1}2}\right)$};


\node at (-17,1.5) {$\left(1 \oplus \tcs{1\phantom{2}\\\phantom{1}2}, \, 0 \right)$};

\end{scope}


\begin{scope}[shift={(0,-12)},scale=0.6]

\coordinate (b0) at (0,0);
\coordinate (b1) at (1.5,-1);
\coordinate (b2) at (3,0.5);
\coordinate (a0) at (0,3);
\coordinate (a1) at (1.5,2);
\coordinate (a2) at (3,3.5);

\draw[fill=red,fill opacity=0.5] (b0) -- (a1) -- (a2) -- (b0);
\draw[fill=blue,fill opacity=0.5] (b0) -- (a1) -- (b2) -- (b0);

\node [left = 1mm of a0] {$a_{0}$};
\node [below right = 1mm of a1] {$a_{1}$};
\node [below right = 1mm of a2] {$a_{2}$};
\node [left = 1mm of b0] {$b_{0}$};
\node [below right = 1mm of b1] {$b_{1}$};
\node [below right = 1mm of b2] {$b_{2}$};

\draw (a0) -- (a1) -- (a2) -- (a0);
\draw[dashed] (b0) -- (b2);
\draw (b0) -- (b1) -- (b2);
\draw (a0) -- (b0);
\draw (a1) -- (b1);
\draw (a2) -- (b2);


\node at (-3,1.5) {$baa, bab$};


\node at (-6,1.5) {$231$};


\node at (-11,1.5) {$\left(\tcs{1\phantom{2}\\\phantom{1}2} \to 0\right) \oplus \left(\tcs{1\phantom{2}\\\phantom{1}2} \to \tcs{\phantom{1}2\\1\phantom{2}}\right)$};


\node at (-17,1.5) {$\left(2, \, \tcs{1\phantom{2}\\\phantom{1}2} \right)$};

\end{scope}


\begin{scope}[shift={(0,-16)},scale=0.6]

\coordinate (b0) at (0,0);
\coordinate (b1) at (1.5,-1);
\coordinate (b2) at (3,0.5);
\coordinate (a0) at (0,3);
\coordinate (a1) at (1.5,2);
\coordinate (a2) at (3,3.5);

\draw[fill=red,fill opacity=0.5] (b0) -- (b1) -- (a2) -- (b0);
\draw[fill=blue,fill opacity=0.5] (a0) -- (b1) -- (a2) -- (a0);

\node [left = 1mm of a0] {$a_{0}$};
\node [below right = 1mm of a1] {$a_{1}$};
\node [below right = 1mm of a2] {$a_{2}$};
\node [left = 1mm of b0] {$b_{0}$};
\node [below right = 1mm of b1] {$b_{1}$};
\node [below right = 1mm of b2] {$b_{2}$};

\draw (a0) -- (a1) -- (a2) -- (a0);
\draw[dashed] (b0) -- (b2);
\draw (b0) -- (b1) -- (b2);
\draw (a0) -- (b0);
\draw (a1) -- (b1);
\draw (a2) -- (b2);


\node at (-3,1.5) {$aba, bba$};


\node at (-6,1.5) {$312$};


\node at (-11,1.5) {$\left(\tcs{\phantom{1}2\\1\phantom{2}} \to \tcs{1\phantom{2}\\\phantom{1}2}\right) \oplus \left(\tcs{\phantom{1}2\\1\phantom{2}} \to 0\right)$};


\node at (-17,1.5) {$\left(1, \, \tcs{\phantom{1}2\\1\phantom{2}} \right)$};

\end{scope}


\begin{scope}[shift={(0,-20)},scale=0.6]

\coordinate (b0) at (0,0);
\coordinate (b1) at (1.5,-1);
\coordinate (b2) at (3,0.5);
\coordinate (a0) at (0,3);
\coordinate (a1) at (1.5,2);
\coordinate (a2) at (3,3.5);

\draw[fill=red,fill opacity=0.5] (b0) -- (a1) -- (a2) -- (b0);
\draw[fill=blue,fill opacity=0.5] (b0) -- (b1) -- (a2) -- (b0);

\node [left = 1mm of a0] {$a_{0}$};
\node [below right = 1mm of a1] {$a_{1}$};
\node [below right = 1mm of a2] {$a_{2}$};
\node [left = 1mm of b0] {$b_{0}$};
\node [below right = 1mm of b1] {$b_{1}$};
\node [below right = 1mm of b2] {$b_{2}$};

\draw (a0) -- (a1) -- (a2) -- (a0);
\draw[dashed] (b0) -- (b2);
\draw (b0) -- (b1) -- (b2);
\draw (a0) -- (b0);
\draw (a1) -- (b1);
\draw (a2) -- (b2);


\node at (-3,1.5) {$baa, bba$};


\node at (-6,1.5) {$321$};


\node at (-11,1.5) {$\left(\tcs{1\phantom{2}\\\phantom{1}2} \to 0\right) \oplus \left(\tcs{\phantom{1}2\\1\phantom{2}} \to 0\right)$};


\node at (-17,1.5) {$\left(0, \, \tcs{1\phantom{2}\\\phantom{1}2} \oplus \tcs{\phantom{1}2\\1\phantom{2}}\right)$};

\end{scope}


\end{tikzpicture}
\]
\end{figure}

\begin{proposition}\label{prop:flip}
Under the bijection between $\twosilt\preproj$ and $\tri(\prism)$, mutations of two-term silting complexes correspond to bistellar flips of triangulations.
\end{proposition}
\begin{proof}
Since two-term silting complexes are related by mutation if and only if they differ by only one indecomposable summand, it follows from Corollary~\ref{cor:triang_bij} that two silting complexes are related by a mutation if and only if the corresponding triangulations differ by only one codimension one internal simplex. In turn, it is then true that two triangulations $\mathcal{T}$ and $\mathcal{T}'$ are bistellar flips of each other if and only if they differ by an internal $n$-simplex. Indeed, if $\mathcal{T}$ and $\mathcal{T}'$ are two triangulations, with $\mathcal{S}$ the subdivision given by their common internal $n$-simplices, then $\mathcal{S}$ is an almost triangulation if and only if it contains $n - 1$ internal $n$-simplices.
\end{proof}

\subsection{Compatibility with permutations}

In this section, we show that our bijection between the support $\tau$-tilting pairs $\sttilt\preproj$ and triangulations $\tri(\prism)$ is compatible with the existing bijections between $\symm$ and $\sttilt\preproj$ and between $\symm$ and $\tri(\prism)$.

\begin{proposition}
For all $w \in \symm$, we have that $\mathcal{T}(I_{w}, P_{w}) = \mathcal{T}_{w}$.
\end{proposition}
\begin{proof}
We show this by induction on the length of $w$. The base case is the identity permutation $e$. We have that $\mathcal{T}_{e}$ has $(n + 1)$-simplices \[\{\{a_{0}, \dots, a_{j}, b_{j}, \dots, b_{n}\} \st 0 \leqslant j \leqslant n\}.\] The internal $n$-simplices of this triangulation are \[\{ \{a_{0}, \dots, a_{j}, b_{j + 1}, \dots, b_{n}\} \st 0 \leqslant j < n \}.\] We have that $(I_{e}, P_{e}) = (\preproj, 0)$, and by Proposition~\ref{prop:indec_bij}, the word corresponding to the indecomposable projective $P_{j}$ is \[\underbrace{a \dots a}_{j}\underbrace{b \dots b}_{n - j}.\] These are the words of the internal $n$-simplices of $\mathcal{T}_{e}$, so we have $\mathcal{T}_{e} = \mathcal{T}(I_{e}, P_{e})$.

We now suppose that we have $w \in \symm$ such that $w = w's_{i}$ for some $i$, so that $I_{w} = I_{w'}I_{i}$. Hence, $I_{w}$ is obtained from $I_{w'}$ by removing composition factors in the top given by $S_{i}$, the simple $\preproj$-module at the vertex $i$. Given an indecomposable summand $M'$ of $I_{w'}$, we have, by Proposition~\ref{prop:g_vec} that $S_{i}$ occurs in the top of $M'$ if and only if $x_{i - 1}x_{i} = ab$ for $\word{M'} = x_{0}x_{1} \dots x_{n}$. Then, if $M$ is the corresponding indecomposable summand of $I_{w}$, we have that $\word{M} = x_{0}x_{1} \dots x_{i - 2}bax_{i + 1} \dots x_{n}$.

On the other hand, considering the permutations, if we let $w' = k_{0} \dots k_{j - 1}(i - 1)k_{j + 1} \dots k_{l - 1}ik_{l + 1} \dots k_{n}$, then $w = k_{0} \dots k_{j - 1}ik_{j + 1} \dots k_{l - 1}(i - 1)k_{l + 1} \dots k_{n}$. Note that, by assumption, we have $j < l$, since the length of $w$ is greater than the length of $w'$. Using the description of $\mathcal{T}_{w}$ and $\mathcal{T}_{w'}$ from Section~\ref{sect:back:triang:perm}, we have that if $X = x_{0}x_{1} \dots x_{n}$ is the word of a simplex in $\mathcal{T}_{w'}$, then $\Delta_X$ is a simplex of $\mathcal{T}_{w}$ too if and only if $x_{i - 1}x_{i} = aa$ or $x_{i - 1}x_{i} = bb$. Furthermore, if $x_{i - 1}x_{i} = ab$, then $x_{0} \dots x_{i - 2}bax_{i + 1} \dots x_{n}$ is the word of a simplex in $\mathcal{T}_{w}$. The case $x_{i - 1}x_{i} = ba$ is not possible since $i - 1$ precedes $i$ in $w$. Comparing with the previous paragraph, we see that the summands of $(I_{w'}, P_{w'})$ which change to give $(I_{w}, P_{w})$ correspond to the internal $n$-simplices of $\mathcal{T}_{w'}$ which change to give $\mathcal{T}_{w}$. Moreover, the change in the words of the simplices corresponds precisely to the change in the upper contours of the summands of the $\tau$-tilting pair. Hence, we obtain that $\mathcal{T}(I_{w}, P_{w}) = \mathcal{T}_{w}$ and the result follows by induction.
\end{proof}

\begin{remark}
Note that our results therefore also give a different way of obtaining the support $\tau$-tilting pair over $\preproj$ corresponding to a permutation in $\symm$ to the description from \cite{mizuno-preproj}. Namely, given a permutation, one uses the description of the corresponding triangulation of $\prism$ from Section~\ref{sect:back:triang:perm} to obtain a set of words in $\abseq$. One then uses Proposition~\ref{prop:indec_bij} to translate this into a support $\tau$-tilting pair over $\preproj$ by using these words to give the upper contours of the indecomposable $\tau$-rigid summands. This does not require a reduced expression for the permutation, whereas the description from \cite{mizuno-preproj} does.
\end{remark}

\bibliographystyle{alpha}
\bibliography{biblio}

\end{document}